\documentclass[preprint,11pt,3p]{elsarticle}

\usepackage{times}
\usepackage{color}
\usepackage[colorlinks=true]{hyperref}
\usepackage{amsmath,amssymb,amsthm,amscd,mathrsfs}
\usepackage{setspace}
\usepackage{graphicx} %插入图片的宏包
\usepackage{float} %设置图片浮动位置的宏包
\usepackage{subfigure} %插入多图时用子图显示的宏包
\usepackage{tikz}
\numberwithin{equation}{section}

\newtheorem{theorem}{Theorem}[section]

\newtheorem{proposition}{Proposition}[section]
\newtheorem{lemma}{Lemma}[section]

\newtheorem{corollary}{Corollary}[section]
\newtheorem{remark}{Remark}[section]

\biboptions{sort&compress}

\journal{Elsevier}

\begin{document}

\begin{frontmatter}

\title{Metrical theory of power-2-decaying Gauss-like expansion}
\author[label1,label2]{Zhihui Li}
\ead{huizhili123456@126.com}
\author[label1]{Xin Liao}
\ead{xin_liao@whu.edu.cn}
\author[label1]{Dingding Yu\corref{cor1}}
\ead{yudding_sgr@whu.edu.cn}

\address[label1]{School of Mathematics and Statistics, Wuhan University, Wuhan 430072, China}
\address[label2]{Wuhan business university, Wuhan 430056, China}

\cortext[cor1]{corresponding author}

\begin{abstract}
Each $x\in (0,1]$ can be uniquely expanded as a power-2-decaying Gauss-like expansion, in the form of
\begin{equation*}
x=\sum_{i=1}^{\infty}2^{-(d_1(x)+d_2(x)+\cdots+d_i(x))},\qquad d_i(x)\in \mathbb{N}.
\end{equation*}
Let $\phi:\mathbb{N}\to \mathbb{R}^{+}$ be an arbitrary positive function. We are interested in the size of the set
$$F(\phi)=\{x\in (0,1]:d_n(x)\ge \phi(n)~~\text{for infinity many}~n\}.$$
We prove a Borel-Bernstein theorem on the zero-one law of the Lebesgue measure of $F(\phi)$. When the Lebesgue measure of $F(\phi)$ is zero, we calculate its  Hausdorff dimension. Furthermore, we analyse the growth rate of the maximal digit among the first $n$ digits from probability and multifractal perspectives.

\end{abstract}
\begin{keyword}
power-2-decaying Gauss-like expansion; Borel-Bernstein theorem; maximal digits; Hausdorff dimension
\end{keyword}
\end{frontmatter}

\section{Introduction}

Each real number $x\in (0,1]$ can be uniquely expanded as a power-2-decaying Gauss-like expansion (P2DGL expansion, for short), in the form:
\begin{equation}\label{def}
x=\sum_{i=1}^{\infty}2^{-(d_1(x)+d_2(x)+\cdots+d_i(x))},
\end{equation}
where $d_i(x)\in \mathbb{N}$. Further results about P2DGL expansion can be found in \cite{neunhauserer2011hausdorff} and \cite{neunhauserer2022dimension}, where the P2DGL expansion is called base 2 expansion. 

Motivated by the metric theory of the regular continued fraction expansion, for a given function $\phi:\mathbb{N}\to \mathbb{R}^{+}$, we study the Lebesgue measure and Hausdorff dimension of the set
\begin{equation}\label{1111}
 F(\phi)=\{x\in (0,1]:d_n(x)\ge \phi(n),~~\text{i.m.}~n\},
\end{equation}
where i.m. denotes `infinitely many'.  Concerning the Lebesgue measure $\mathcal{L}$, we have the following theorem which is a version of Borel-Bernstein theorem for P2DGL expansion.

\begin{theorem}\label{main theorem BB}
Let $\phi:\mathbb{N}\to \mathbb{R}^{+}$ be an arbitrary positive function and $F(\phi)$ be defined as in (\ref{1111}).
Then $\mathcal{L}(F(\phi))$ is null or full according as whether the series $\sum_{n=1}^{\infty}\frac{1}{2^{\phi(n)}}$ converges or not.
\end{theorem}

By Theorem \ref{main theorem BB}, many sets of points with restrictions on their digits typically have null Lebesgue measure. To further analyse the size and complexity of such null Lebesgue measure sets, one uses the notion of Hausdorff dimension. Neunh\"{a}userer \cite{neunhauserer2011hausdorff} proved that the set
$$\{x\in (0,1]: \text{the sequence $\{d_i(x)\}_{i\in \mathbb{N}}$ is bounded}\}$$
is of Lebesgue measure 0, but of Hausdorff dimension 1. Later in \cite{neunhauserer2022dimension}, the same author proved that
$$\dim_{\mathrm{H}}\{x\in(0,1]:\lim_{i\to \infty}d_i(x)=\infty\}=0,$$
where $\dim_{\mathrm{H}}$ stands for the Hausdorff dimension.
Parallel to the famous result of Wang--Wu \cite{wang2008hausdorff} on the continued fractions, we obtain the Hausdorff dimension of $F(\phi)$.

\begin{theorem}\label{main theorem im}
  Let $\phi:\mathbb{N}\to \mathbb{R}^{+}$ be an arbitrary positive function and $F(\phi)$ be defined as in (\ref{1111}).
 Denote $\alpha:=\liminf\limits_{n\to \infty}\frac{\phi(n)}{n}$. Then
  \begin{equation}
  \dim_{\mathrm{H}}F(\phi)=\begin{cases}
                             s(\alpha), & \mbox{if } 0\le \alpha<\infty \\
                             0, & \mbox{if } \alpha=+\infty,
                           \end{cases}
  \end{equation}
  where $s(\alpha)$ is the unique real solution of the equation $\sum\limits_{k=1}^{\infty}\left(\frac{1}{2^{\alpha}2^k}\right)^s=1$.
\end{theorem}

The image of the function $\alpha\mapsto s(\alpha)$ is illustrated in Figure 1.

  \begin{figure}[H]
\centering
\includegraphics[width=0.5\textwidth]{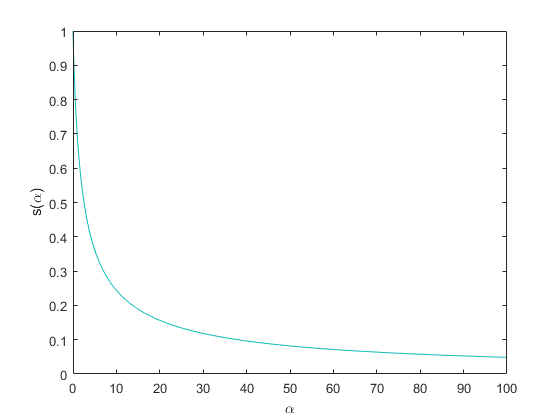}
\caption{The image of the function $\alpha\mapsto s(\alpha)$.}
\end{figure}

 To deepen our understanding of P2DGL expansions, we also investigate the maximal digits in P2DGL expansions. Define
$$L_n(x)=\max\{d_i(x):1\le i\le n\},$$
as the maximal digit among the first $n$ digits in the P2DGL expansion of $x$. 

There have been extensive studies and applications of the maximal digits for the continued fractions and the L{\"u}roth expansions. See \cite{philipp1976conjecture, wu2009distribution, liao2016subexponentially, fang2021largest} (respectively, \cite{galambos1976rate, shen2014note, song2017level}) and references therein for more
results on the maximal digits for the continued fraction (respectively, L{\"u}roth expansions).

The following theorem describes the growth rate of $L_n(x)$ in P2DGL expansion for typical $x$. 

\begin{theorem}\label{main theorem 1}
  For almost all $x\in (0,1]$,
  $$\lim\limits_{n\to \infty} \frac{L_n(x)}{\log_2 n}=1.$$
\end{theorem}
This indicates that for almost all $x$, the largest digits of $x$ tends to infinity with a logarithmic rate. However, it is evident that there exist points for which the digits can be arbitrarily large. This prompts the question of whether it is possible for some points $x$ that $L_n(x)$ can tend to infinity with other given orders. 

Let $r>0$ and $\alpha>0$. We are interested in the Hausdorff dimension of the sets
$$E(r,\alpha)=  \left\{x\in (0,1]:\lim_{n\to \infty}\frac{L_n(x)}{n^r}=\alpha\right\},$$
and $$\overline{E}(r,\alpha)=  \left\{x\in (0,1]:\limsup_{n\to \infty}\frac{L_n(x)}{n^r}=\alpha\right\}.$$
We have the following results.
\begin{theorem}\label{main theorem 2}
  Let $r>0$ and $\alpha >0$. Then 
  $$
    \dim_{\mathrm{H}}E(r,\alpha)=\begin{cases}
                                            1, & \mbox{if}~~r< 1 \\
                                            0, & \mbox{if}~~r\ge 1.
                                          \end{cases}
  $$
\end{theorem}
\begin{corollary}\label{coro1.1}
  Let $\alpha >0$. Then
  $$\dim_{\mathrm{H}}\Big\{x\in (0,1]:\lim_{n\to \infty}\frac{L_n(x)}{\log_2 n}=\alpha\Big\}=1.$$
\end{corollary}

\begin{theorem}\label{main theorem 3}
  Let $r>0$, $\alpha >0$ and $s(\alpha)$ be defined as in Theorem \ref{main theorem im}. Then
  $$
    \dim_{\mathrm{H}}\overline{E}(r,\alpha)=\begin{cases}
                                            1, & \mbox{if}~~r< 1 \\
                                            s(\alpha), & \mbox{if}~~r= 1 \\
                                            0, & \mbox{if}~~r> 1.
                                          \end{cases}
  $$
\end{theorem}
Theorems \ref{main theorem 2} and \ref{main theorem 3} show that when $0<r<1$, $E(r,\alpha)$ and $\overline{E}(r, \alpha)$ have Hausdorff dimension of 1. However, when $r\geq 1$, both $E(r,\alpha)$ and $\overline{E}(r, \alpha)$ have Hausdorff dimension of 0. It is interesting that the dimension of $E(r,\alpha)$ undergoes a jump from 1 to 0 at $r=1$, while an intriguing change in the Hausdorff dimension of $\overline{E}(r, \alpha)$ occurs when $r = 1$.

\subsection*{Structure of the paper:}
\begin{enumerate}[$\bullet$]
  \item In the next section, we compile some basic facts of the P2DGL expansion which will be used later.
  \item  We prove Theorems \ref{main theorem BB} and \ref{main theorem 1} in Sections 3 and 4.
  \item Section 5 provides a detailed proof of Theorem  \ref{main theorem im}.
  \item  Sections 6 and 7 are devoted to proving Theorems \ref{main theorem 2} and \ref{main theorem 3}.
\end{enumerate}

Throughout the paper we denote by $|\cdot|$ the diameter of a subset of $I$, by $\lfloor x \rfloor$ the integer part of $x$, by $\lceil x\rceil$ the smallest integer larger than $x$, by $\#$ the cardinality of a set and by $\mathcal{H}^s$ the $s$-dimensional Hausdorff measure.

\section{Preliminaries}
In this section, we recall some fundamental results on P2DGL expansion, and give several lemmas crucial for the subsequent proofs.

Define an interval map $T:(0,1]\to (0,1]$ as
 \begin{equation}\label{T}
   Tx=2^nx-1,\qquad \text{if}~x\in (1/2^n,1/2^{n-1}],~~n\in \mathbb{N}.
 \end{equation}
 The map $T$ is illustrated in Figure 2.

 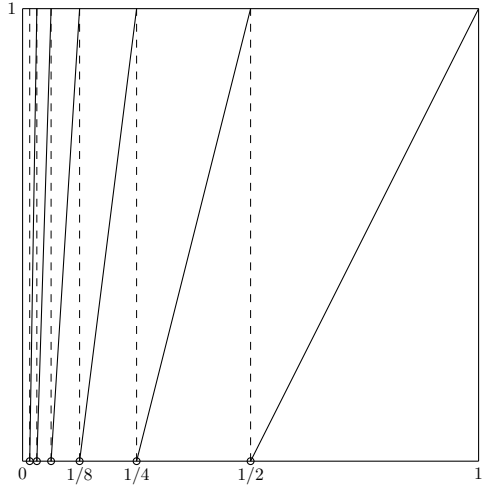
\begin{figure}[h]
   \centering
\begin{tikzpicture}[scale=1.5]
\tikzstyle{every node}=[scale=0.6]
\draw[black] (0,0) -- (0,4);
\draw[black] (4,0) -- (4,4);
\draw[black] (0,4) -- (4,4);
\draw[black] (0,0) -- (4,0);
\draw[black] (2,0) -- (4,4);
\draw[dashed,very thin] (2,0) -- (2,4);
\draw (2,0) circle (0.03);
\draw[black] (1,0) -- (2,4);
\draw[dashed,very thin] (1,0) -- (1,4);
\draw (1,0) circle (0.03);
\draw[black] (1/2,0) -- (1,4);
\draw[dashed,very thin] (1/2,0) -- (1/2,4);
\draw (1/2,0) circle (0.03);
\draw[black] (1/4,0) -- (1/2,4);
\draw[dashed,very thin] (1/4,0) -- (1/4,4);
\draw (1/4,0) circle (0.03);
\draw[black] (1/8,0) -- (1/4,4);
\draw[dashed,very thin] (1/8,0) -- (1/8,4);
\draw (1/8,0) circle (0.03);
\draw[black] (1/16,0) -- (1/8,4);
\draw[dashed,very thin] (1/16,0) -- (1/16,4);
\draw (1/16,0) circle (0.03);
\coordinate[label=below:$1$](1) at (4,0);
\coordinate[label=below:$1/2$](1/2) at (2,0);
\coordinate[label=below:$1/4$](1/4) at (1,0);
\coordinate[label=below:$1/8$](1/8) at (1/2,0);
\coordinate[label=below:$0$](0) at (0,0);
\coordinate[label=left:$1$](1) at (0,4);

\end{tikzpicture}
\caption{The map $T$}
 \end{figure} 
  It can be checked that the digits $d_i(x)$ in (\ref{def}) are determined by the formula
 \begin{equation}\label{dT}
   d_i(x)=\lceil -\log_2(T^{i-1}(x)) \rceil,\qquad \forall i\ge 1.
 \end{equation}

For $n\in \mathbb{N}$, and $(d_1,d_2,\cdots,d_n)\in \mathbb{N}^n$, we call
\begin{equation*}
  I_n(d_1,d_2,\cdots,d_n)=\{x\in (0,1]: d_1(x)=d_1, d_2(x)=d_2, \cdots, d_n(x)=d_n\}
\end{equation*}
a basic interval of rank $n$. The basic interval of rank $n$ containing $x$ is denoted by $I_n(x)$.  By the similarity of $T$ to the classical Gauss map, and the decaying rate of the associated basic intervals of rank $1$, we call (\ref{def})  power-2-decaying Gauss-like (P2DGL) expansion. Note that for any $x\in (0,1]$, we have $T^nx\neq 0$ for any $n\ge 1$, and
\begin{equation*}
  x=\sum_{i=1}^{n}2^{-(d_1(x)+d_2(x)+\cdots+d_i(x))}+\frac{T^{n}x}{2^{d_1(x)+d_2(x)+\dots+d_n(x)}},\qquad \forall n\in \mathbb{N}.
\end{equation*}

Then we can easily obtain the following lemma which gives the length of a basic interval.
\begin{lemma}\label{Lemma endpoint}
  The endpoints of $I_n(d_1,d_2,\cdots,d_n)$ are $\sum\limits_{i=1}^{n}\frac{1}{2^{d_1+\cdots+d_i}}$ and $\sum\limits_{i=1}^{n-1}\frac{1}{2^{d_1+\cdots+d_i}}+\frac{1}{2^{d_1+\dots+d_n-1}}.$ Consequently,
  \begin{equation}\label{|I_n|}
 |I_n(d_1,d_2,\cdots,d_n)|=\frac{1}{2^{d_1+\cdots+d_n}}\quad \text{and}\quad \frac{|I_{n+1}(d_1,d_2,\cdots,d_n,k)|}{|I_n(d_1,d_2,\cdots,d_n)|}=2^{-k}.
\end{equation}
\end{lemma}
\begin{proof}
  It can be checked by the definition of $I_n(d_1,d_2,\cdots,d_n)$.
\end{proof}

Furthermore, we can verify that $T$ is a measure-preserving transformation, and all the $d_i(\cdot)$ are i.i.d. random variables on the probability space $\left((0,1],\sigma\left((0,1]\right), \mathcal{L}\right)$.
\begin{lemma}
  The transformation $T$ defined by (\ref{T}) preserves the Lebesgue measure $\mathcal{L}$.
\end{lemma}
\begin{proof}
  Noting that for any $(a,b]\subset (0,1]$,
$$T^{-1}(a,b]=\bigcup_{n=1}^{\infty}\left(\frac{a+1}{2^n},\frac{b+1}{2^n}\right],$$
we have
$$\mathcal{L}(T^{-1}(a,b])=\sum_{n=1}^{\infty}\frac{b-a}{2^n}=b-a=\mathcal{L}\big((a,b]\big).$$
By Sierpiński–Dynkin's $\pi$-$\lambda$ theorem, the proof is finished.
\end{proof}

\begin{lemma}\label{iid}
  For any $i\ge 1$, the digits functions $x \mapsto d_i(x)$ defined by (\ref{dT}) are independent and identically distributed.
\end{lemma}
\begin{proof}
 For any $i\ge 1$, we have
\begin{align*}
  &\mathcal{L}(\{x\in (0,1]:d_i(x)=d_i\}) \\ = & \mathcal{L}\left(\bigcup_{d_1=1}^{\infty}\cdots \bigcup_{d_{i-1}=1}^{\infty}\{x\in (0,1]:d_1(x)=d_1, \cdots,d_{i-1}(x)=d_{i-1},d_n(x)=d_i\}\right) \\
    =&\sum_{d_1=1}^{\infty}\cdots \sum_{d_{i-1}=1}^{\infty}\mathcal{L}(\{x\in (0,1]:d_1(x)=d_1,\cdots, d_{i-1}(x)=d_{i-1},d_i(x)=d_i\}) \\
    =&\sum_{d_1=1}^{\infty}\cdots \sum_{d_{i-1}=1}^{\infty}\frac{1}{2^{d_1+\cdots+d_{i-1}+d_i}}=\frac{1}{2^{d_i}}.
\end{align*}
Noting that for all $n\ge 1$ and $1\le i\le n$,
$$\sum_{d_i=1}^{\infty}\frac{1}{2^{d_1+\dots+d_n}}=
\frac{1}{2^{\sum_{j=1}^{i-1}d_j+\sum_{j=i+1}^{n}d_j}},$$
we deduce that for any $k\in \mathbb{N}$ and $n_1<n_2<\cdots <n_k$,
\begin{align*}
  &\mathcal{L}(\{x \in (0,1]:d_{n_1}(x)=d_{n_1}, d_{n_2}(x)=d_{n_2},\cdots,d_{n_k}(x)=d_{n_k}\}) \\
  =&\frac{1}{2^{d_{n_1}+\cdots+d_{n_k}}}=\prod_{i=1}^k \mathcal{L}(\{x\in (0,1]:d_{n_i}(x)=d_{n_i}\}).
\end{align*}
Thus the digits functions are independent and identically distributed.
\end{proof}

 The Hausdroff dimension of the set of numbers whose digits are bounded is given by the following lemma.
 
  For any integer $M\ge 2$, let
\begin{equation}\label{EM}
  E_M=\{x\in (0,1]: 1\le d_n(x)\le M,~\forall n\ge 1\}.
\end{equation}
 
\begin{lemma}(\cite[Proposition 3.1]{neunhauserer2022dimension})\label{lemmasM1}
  For any $M\ge 2$, we have $\dim_{\mathrm{H}}E_M=\log_2(s_M)$, where $s_M\in (1,2)$ is the unique real solution of
  $$x^M-x^{M-1}-\cdots-x^2-x-1=0,$$
  and
\begin{equation*}
  \lim_{M\to \infty}s_M=2.
\end{equation*}
\end{lemma}

For a given function $\phi:\mathbb{N}\to\mathbb{R}^+$, define
\begin{equation}\label{Gphi}
  G(\phi)=\{x\in (0,1]:L_n(x)\ge \phi(n),~\text{i.m.}~n\}.
\end{equation}
The following lemma establishes the relationship between $F(\phi)$ and $G(\phi)$.
\begin{lemma}\label{use3}
  Let $\phi:\mathbb{N}\to\mathbb{R}^+$ be an increasing positive function with $\lim\limits_{n\to \infty}\phi(n)=\infty$. Let $F(\phi)$ and $G(\phi)$ be defined as in (\ref{1111}) and (\ref{Gphi}). Then
  $$G(\phi)\subset F(\phi).$$
\end{lemma}
\begin{proof}
For any $x\in G(\phi)$,  $L_n(x)\ge\phi(n)$ occurs infinitely often. Since $\phi$ is increasing, then there exists an $m_n\le n$ such that $$d_{m_n}(x)\ge \phi(n)\ge \phi(m_n).$$
  If the set $\{m_n:n\in \mathbb{N}\}$ has only finitely many elements, then there exists $M: =\max\{d_{m_n}: n\in \mathbb{N}\}$ such that $M>\phi (n)$. This contradicts the assumption $\lim\limits_{n\to \infty}\phi(n)=\infty$.
  Therefore, $d_n(x)\ge \phi(n)$ occurs infinitely often.
\end{proof}

Consider a sequence $\{X_i\}_{i\in \mathbb{N}}$ of i.i.d. random variables defined on a probability space $(\Omega, \mathscr{A}, \mathbb{P})$, with a common distribution function $F(x)=\mathbb{P}({X_n\le x})$ for all $x\in (-\infty, +\infty)$ and $n\in \mathbb{N}$. The following theorem of Barndorff-Nielsen \cite{barndorff1961rate} will be useful for us.
\begin{theorem}(\cite[Theorem 1]{barndorff1961rate})\label{theorem OBN}
  Let $\{\lambda_n\}$ be a non-decreasing sequence of real numbers, such that the sequence $\{(F(\lambda_n))^n\}$ is non-increasing. Then
  $$\mathbb{P}\left(\left\{\max_{1\le i\le n}X_i\le \lambda_n~~ \text{for infinitely many}~n\right\}\right)=\left\{
  \begin{aligned}
     & 0 \\
     & 1,
  \end{aligned}
  \right.
  $$
  if
  $$\sum_{n=1}^{\infty}(F(\lambda_n))^n\frac{\log\log n}{n}
  \left\{
  \begin{aligned}
     & <\infty \\
     & =\infty.
  \end{aligned}
  \right.
  $$
\end{theorem}

\section{Proof of Theorem \ref{main theorem BB}}

\begin{proof}
  We begin with the divergence part. Suppose $I_{m+n}$ is a basic interval of rank $m+n$ in which the points satisfy the conditions
  \begin{equation}\label{63}
    d_{m+\ell}(x)<\phi(m+\ell), \qquad \ell=1,2,\cdots,n.
  \end{equation}
  The points of the interval $I_{m+n}$ that satisfy the additional condition $d_{m+n+1}=k$ form a basic interval of rank $m+n+1$, denoted by $I_{m+n+1}^{(k)}$. By (\ref{|I_n|}), we have
$|I_{m+n+1}^{(k)}|=\frac{1}{2^k}|I_{m+n}|.$
Hence,
\begin{align*}
  \sum_{k\ge \phi(m+n+1)}|I_{m+n+1}^{(k)}|&=|I_{m+n}|\sum_{k\ge \phi(m+n+1)}\frac{1}{2^k}\geq|I_{m+n}|\sum_{\ell=1}^{\infty}\frac{1}{2^{\phi(m+n+1)+\ell}}= |I_{m+n}|\frac{1}{2^{\phi(m+n+1)}}.
\end{align*}
Since
$\sum\limits_{k=1}^{\infty}|I_{m+n+1}^{(k)}|=|I_{m+n}|,$
it follows that
\begin{equation}\label{ineq1}
  \sum_{k< \phi(m+n+1)}|I_{m+n+1}^{(k)}|<\left(1-\frac{1}{2^{\phi(m+n+1)}}\right)|I_{m+n}|.
\end{equation}
Define
  $$J_{m,n}=\bigcup_{d_{m+\ell}<\phi(m+\ell),~\ell=1,2,\cdots,n}I_{m+n}(d_1,\cdots,d_m,d_{m+1},\cdots,d_{m+n}).$$
By summing the inequality (\ref{ineq1}) over all basic intervals of rank $m+n$ that satisfy the conditions (\ref{63}), we obtain
\begin{align}\label{ineq2}
\mathcal{L}(J_{m,n+1})&\nonumber=\sum_{d_{m+\ell}<\phi(m+\ell)\atop \ell=1,2,\cdots,n}\sum_{k<\phi(m+n+1)}|I_{m+n+1}^{(k)}|\\&< \left(1-\frac{1}{2^{\phi(m+n+1)}}\right)\sum_{d_{m+\ell}<\phi(m+\ell)\atop \ell=1,2,\cdots,n}|I_{m+n}|=\left(1-\frac{1}{2^{\phi(m+n+1)}}\right)\mathcal{L}(J_{m,n}).
\end{align}
Iterating the inequality (\ref{ineq2}) gives
$$\mathcal{L}(J_{m,n})<\mathcal{L}(J_{m,1})\prod_{\ell=2}^{n}\left(1-\frac{1}{ 2^{\phi(m+\ell)}}\right).$$
If the series $\sum\limits_{n=1}^{\infty}\frac{1}{2^{\phi(n)}}$ diverge, then for arbitrary $m$, the series
$\sum\limits_{\ell=2}^{\infty}\frac{1}{ 2^{\phi(m+\ell)}}$
 also diverge. Hence, it follows that the product
$$\prod_{\ell=2}^{n}\left(1-\frac{1}{ 2^{\phi(m+\ell)}}\right)$$
approaches to zero as $n \to \infty$. Therefore, for arbitrary $m$,
$\mathcal{L}(J_{m,n})\to 0 \ \text{as} \ n\to\infty.$

Let $$E_m:=\{x\in (0,1]:d_{m+\ell}(x)<\phi(m+\ell),~\forall \ell\ge 1\}.$$
Then for any $n\ge 1$, we have $E_m\subset J_{m,n}$. Hence $\mathcal{L}(E_m)=0$. Finally, let
$E=\bigcup\limits_{m=1}^{\infty} E_m,$
then $\mathcal{L}(E)=0$. Since the complement of the set $E$ is $F(\phi)$, this proves the first assertion of the theorem.

Suppose now the series $\sum\limits_{n=1}^{\infty}\frac{1}{2^{\phi(n)}}$ converge. Assume that $I_n$ is one of the basic intervals of rank $n$ and that $I_{n+1}^{(k)}$ is its subinterval of rank $n+1$ with $d_{n+1}(x)=k$, for any $x\in I_{n+1}^{(k)}$. By the equality (\ref{|I_n|}), we have $\frac{|I_{n+1}^{(k)}|}{|I_n|}=2^{-k}$, and
\begin{equation*}
  \sum_{k\ge\phi(n+1)} |I_{n+1}^{(k)}|  < |I_n| \sum_{k\ge \phi(n+1)}\frac{1}{2^k}
    \le |I_n| \sum_{\ell=0}^{\infty}\frac{1}{2^{\phi(n+1)+\ell-1}}
   =|I_n|\frac{1}{2^{\phi(n+1)-1}}.
\end{equation*}

Let
$$F_n:=\{x\in (0,1]:d_n(x)\ge \phi(n)\}.$$
We have
$$\mathcal{L}(F_{n+1})=\sum_{d_1,\cdots,d_n=1}^{\infty}\sum_{k\ge\phi(n+1)} |I_{n+1}^{(k)}|<\sum_{d_1,\cdots,d_n=1}^{\infty}|I_n|\frac{1}{2^{\phi(n+1)-1}}= \frac{1}{2^{\phi(n+1)-1}}.$$
Thus, the Lebesgue measure of the sets $F_1,F_2,\dots,F_n,\dots$ form a convergent series. One has $F(\phi)=\bigcap\limits_{N\ge 1}\bigcup\limits_{n\ge N}F_n$. By the Borel-Cantelli Lemma, we have
$\mathcal{L}(F(\phi))=0.$ This proves the convergence part of the theorem.

\end{proof}

\section{Proof of Theorem \ref{main theorem 1}}

 First, we establish several key consequences derived from Theorem  \ref{main theorem BB}. 
\begin{proposition}\label{coro1}
  Let $\phi:\mathbb{N}\to\mathbb{R}^+$ be an increasing positive function and $G(\phi)$ be defined as in (\ref{Gphi}). Then $\mathcal{L}(G(\phi))$ is null or full according as whether the series $\sum_{n=1}^{\infty}\frac{1}{2^{\phi(n)}}$ converges or not.
\end{proposition}
\begin{proof}
  If  $\phi$ is bounded, then the series $\sum_{n=1}^{\infty}\frac{1}{2^{\phi(n)}}$ diverges. Theorem \ref{main theorem BB} implies that, for almost all $x\in (0,1]$,
   the inequality $L_n(x)\ge d_n(x)> \phi(n)$ holds for infinitely many $n$'s.

  Now let $\lim\limits_{n\to \infty}\phi(n)=\infty$. We aim to prove that $d_n(x)\ge\phi(n)$ occurs infinitely often if and only if $L_n(x)\ge\phi(n)$ holds for infinite many values of $n$.
  If $d_n(x)\ge\phi(n)$, it is evident that $L_n(x)\ge \phi(n)$. Conversely, if $L_n(x)\ge \phi(n)$, from Lemma \ref{use3}, we deduce that $d_n(x)\ge\phi(n)$ occurs infinitely often.
\end{proof}

\begin{proposition}\label{coro3}
  For almost all $x\in (0,1]$,
  \begin{align*}
     & \limsup\limits_{n\to \infty}\frac{d_n(x)-\log_2 n}{\log_2\log n}=1, \\
     & \limsup\limits_{n\to \infty}\frac{L_n(x)-\log_2 n}{\log_2\log n}=1.\label{limsup}
  \end{align*}
\end{proposition}
\begin{proof}
For $t>0$, define $\phi_t:\mathbb{N}\to\mathbb{R}^+$ by $\phi_t(n)=\log_2 n+t\log_2\log n$. Then $\sum_{n=1}^{\infty}\frac{1}{2^{\phi_1(n)}}=\infty$. Hence, by Theorem \ref{main theorem BB}, for almost all $x$,
$$d_n(x)\ge \phi_1(n)=\log_2 n +\log_2 \log n \qquad \text{i.m}.~ n.$$
This implies
\begin{equation}\label{11}
  \limsup\limits_{n\to \infty}\frac{d_n(x)-\log_2 n}{\log_2\log n}\ge 1.
\end{equation}
On the other hand, for any $t> 1$, $\sum\limits_{n=1}^{\infty}\frac{1}{2^{\phi_t(n)}}<\infty$. By Theorem \ref{main theorem BB}, we have
$$d_n(x)< \phi_t(n)=\log_2 n +t\log_2 \log n\qquad \text{for all large}~n.$$
This leads to
$$\limsup\limits_{n\to \infty}\frac{d_n(x)-\log_2 n}{\log_2\log n}\le t.$$
By arbitrariness of $t>1$, we have
\begin{equation}\label{12}
  \limsup\limits_{n\to \infty}\frac{d_n(x)-\log_2 n}{\log_2\log n}\le 1.
\end{equation}
Combining (\ref{11}) and (\ref{12}), we obtain
$$\limsup_{n\to \infty}\frac{d_n(x)-\log_2 n}{\log_2\log n}= 1.$$

Applying the same argument, with $d_n(x)$ replaced by $L_n(x)$, from Proposition \ref{coro1}, we deduce that
$$\limsup\limits_{n\to \infty}\frac{L_n(x)-\log_2 n}{\log_2\log n}=1.$$
\end{proof}

While Proposition \ref{coro3} highlights some similarities of $d_n$ and $L_n$, the following proposition  shows that $L_n$ exhibits a "smoother" behavior compared to $d_n$.

\begin{proposition}\label{maintheorem22}
  For almost all $x\in (0,1]$,
  \begin{align*}
     & \liminf\limits_{n\to \infty}\frac{d_n(x)-\log_2 n}{\log_2\log n}=-\infty, \\
     & \liminf\limits_{n\to \infty}\frac{L_n(x)-\log_2 n}{\log_2\log n}=0.
  \end{align*}
\end{proposition}
\begin{proof}
For any integer $t\ge 1$ and $n\ge 1$, it holds that
$$\mathcal{L}(\{x: d_n(x) = t\})=\frac{1}{2^t}.$$
 According to the Borel–Cantelli Lemma, for almost all $x$, $d_n=1$ for infinitely many values of $n$, since $\{d_i\}_{i\in \mathbb{N}}$ is a sequence of  i.i.d. random variables. Therefore,
  $$\liminf\limits_{n\to \infty} d_n=1.$$
  This implies
  $$\liminf\limits_{n\to \infty}\frac{d_n(x)-\log_2 n}{\log_2\log n}=-\infty.$$

Now, let us prove the second equality. By Lemma \ref{iid}, $\{d_i\}_{i\in \mathbb{N}}$ is a sequence of  i.i.d. random variables. Let $F$ be their distribution function.
Then
\begin{align*}
  F(\log_2 n) & =\mathcal{L}(\{d_n(x)\le \log_2 n\})=1-\mathcal{L}(\{d_n(x)>\log_2 n\}) \\
   & =1-\sum_{d_n(x)>\log_2 n}\frac{1}{2^{d_n(x)}}\ge 1-\frac{2}{n}.
\end{align*}
Hence,
$$\sum_{n=1}^{\infty}(F(\log_2 n))^n\frac{\log \log n}{n}\ge \sum_{n=1}^{\infty}\left(1-\frac{2}{n}\right)^n\frac{\log \log n}{n}=\infty.$$
Then, by Theorem \ref{theorem OBN}, we obtain
$$\mathcal{L}(\{L_n(x)\le \log_2 n\quad\text{i.m.}~n\})=1.$$
This implies
\begin{equation}\label{le}
  \liminf_{n\to \infty}\frac{L_n(x)-\log_2 n}{\log_2\log n}\le 0.
\end{equation}
For any $\epsilon>0$, we have
\begin{align*}
  F(\log_2 n+\log_2(\log n)^{-\epsilon}) & =\mathcal{L}(\{d_n(x)\le \log_2 n+\log_2(\log n)^{-\epsilon}\}) \\
   &= 1-\mathcal{L}(\{d_n(x)> \log_2 n+\log_2(\log n)^{-\epsilon}\}) \\
   &=1-\sum_{d_n(x)> \log_2 n+\log_2(\log n)^{-\epsilon}} \frac{1}{2^{d_n(x)}}\\
   &\le 1-\frac{(\log n)^{\epsilon}}{n}.
\end{align*}
Then,
$$\sum_{n=1}^{\infty}(F(\log_2 n+\log_2(\log n)^{-\epsilon}))^n\frac{\log \log n}{n}\le\sum_{n=1}^{\infty}\left(1-\frac{(\log n)^{\epsilon}}{n}\right)^n\frac{\log \log n}{n}<\infty.$$
By Theorem \ref{theorem OBN}, we conclude that
$$\mathcal{L}(\{L_n(x)\le \log_2 n+\log_2(\log n)^{-\epsilon}\quad\text{i.m.}~n\})=0.$$
This implies
\begin{equation}\label{ge}
  \liminf_{n\to \infty}\frac{L_n(x)-\log_2 n}{\log_2\log n}\ge -\epsilon.
\end{equation}
Letting $\epsilon\to 0$,  we deduce from (\ref{le}) and (\ref{ge}) that
$$\liminf_{n\to \infty}\frac{L_n(x)-\log_2 n}{\log_2\log n}=0.$$
\end{proof}

Theorem \ref{main theorem 1} is an immediate consequence of Proposition \ref{coro3} and Proposition \ref{maintheorem22}.

\section{Proof of Theorem \ref{main theorem im}}
Let us present some basic properties of the dimensional number $s(\alpha)$ appearing in Theorem \ref{main theorem im}.
\begin{lemma}\label{Lemma s(B)}
  For $\alpha\ge 0$, let $s(\alpha)$ be the unique real solution of the equation
  \begin{equation*} 
    \sum_{k=1}^{\infty}\left(\frac{1}{2^{\alpha}2^k}\right)^s=1.
  \end{equation*}
  Then
  \begin{enumerate}
    \item[(1)] $s(\alpha)$ is continuous with respect to $\alpha$;
    \item[(2)] $\lim\limits_{\alpha\to \infty}s(\alpha)=0$ and $s(0)=1$.
    \end{enumerate}
\end{lemma}
\begin{proof}
  It is suffices to note that $ \sum\limits_{k=1}^{\infty}\left(\frac{1}{2^{\alpha}2^k}\right)^s=\frac{1}{2^{s\alpha}(2^s-1)}$.
\end{proof}

 We begin by calculating the Hausdorff dimension of
 \begin{equation}\label{add}
   F(\alpha)=\{x\in (0,1]: d_n(x)\ge \alpha n,~~\text{i.m.}~n\},\qquad 0<\alpha<\infty.
 \end{equation}

\begin{theorem}\label{addthm}
  Let $0<\alpha<\infty$ and $F(\alpha)$ be defined as in (\ref{add}). Then
  $$\dim_{\mathrm{H}}F(\alpha)= s(\alpha)$$
\end{theorem}

 The idea of the proof of Theorem \ref{addthm} comes from \cite{wang2008hausdorff} and \cite{shen2017hausdorff}.

\subsection{Upper bound of $\dim_{\mathrm{H}}F(\alpha)$}
The upper bound estimation of $\dim_{\mathrm{H}}F(\alpha)$ is by a natural covering. Note that
$$F(\alpha)=\bigcap_{N\ge 1}\bigcup_{n\ge N}\bigcup_{(d_{1},\cdots,d_n)\in \mathbb{N}^n}J(d_1,\cdots,d_n),$$
 where
$$J(d_1,\cdots,d_n)=\{x\in(0,1]:d_k(x)=d_k, 1\le k\le n, d_{n+1}(x)\ge \alpha(n+1)\}.$$
 We have
 $$\mathcal{L}(J(d_1,\cdots,d_n)) =\sum_{d_{n+1}\ge \alpha(n+1)}|I_{n+1}(d_1,\cdots,d_n,d_{n+1})|\le\frac{1}{2^{d_1+\cdots+d_n}\cdot 2^{ \alpha(n+1)-1}}.$$

Then, for any $\varepsilon >0$,
\begin{align*}
  \mathcal{H}^{s(\alpha)+\varepsilon}(F(\alpha)) & \le \liminf_{N\to \infty}\sum_{n\ge N}\sum_{d_1,\cdots,d_n}| J(d_1,\cdots,d_n)|^{s(\alpha)+\varepsilon} \\
   & \le \liminf_{N\to \infty}\sum_{n\ge N}\sum_{d_1,\cdots,d_n}\left(\frac{1}{2^{d_1+\cdots+d_n}\cdot 2^{\alpha(n+1)-1}}\right)^{s(\alpha)+\varepsilon} \\
   & \le \liminf_{N\to \infty}\sum_{n\ge N}\sum_{d_1,\cdots,d_n}2^{s(\alpha)\varepsilon}\left(\frac{1}{2^{d_1+\cdots+d_n}\cdot 2^{\alpha n}}\right)^{s(\alpha)}\left(\frac{1}{2^{d_1+\cdots+d_n}\cdot 2^{\alpha n}}\right)^{\varepsilon} \\
   & \le \liminf_{N\to \infty}\sum_{n\ge N}\sum_{d_1,\cdots,d_n}2^{s(\alpha)\varepsilon}\left(\frac{1}{2^{d_1+\cdots+d_n}\cdot 2^{\alpha n}}\right)^{s(\alpha)}\left(\frac{1}{2^n\cdot 2^{\alpha n}}\right)^{\varepsilon} \\
   & = \liminf_{N\to \infty}\sum_{n\ge N}2^{s(\alpha)\varepsilon}\left(\frac{1}{2^n\cdot 2^{\alpha n}}\right)^{\varepsilon}=0,
\end{align*}
where the penultimate equality is by the equation (\ref{Lemma s(B)}). Therefore, $\dim_{\mathrm{H}}F(\alpha)\le s(\alpha)+\varepsilon$. Due to the arbitrariness of $\varepsilon$, we conclude that $\dim_{\mathrm{H}}F(\alpha)\le s(\alpha)$.

\subsection{Lower bound of $\dim_{\mathrm{H}}F(\alpha)$}
To establish the lower bound of $\dim_{\mathrm{H}}F(\alpha)$, we will construct a subset $F_{M}(\alpha)\subset F(\alpha)$ and use the Hausdorff dimension of $F_{M}(\alpha)$ to approximate that of $F(\alpha)$.

\subsubsection{Construction of subset $F_M(\alpha)$ of $F(\alpha)$}
For $\alpha>0$,  we define a sequence $\{n_k\}_{k\in \mathbb{N}}$ satisfying the following conditions:
\begin{equation}\label{n_k}
  \lfloor\alpha (1+1/k)n_k\rfloor-\lfloor  \alpha n_k \rfloor> 1\qquad \text{and}\qquad n_1+n_2+\cdots+n_k<\frac{1}{k+1}n_{k+1}.
\end{equation}
Let $M\in \mathbb{N}$. We define a subset $F_{M}(\alpha )\subset F(\alpha )$ as
\begin{align}\label{same}
  F_{M}(\alpha )=\{ &\nonumber x\in (0,1]:\lfloor \alpha n_k\rfloor+1\le d_{n_k}(x)\le \lfloor(1+1/k)\alpha n_k\rfloor,~\text{for all}~ k\ge 1, \\
   & \text{and}~1\le d_j(x)\le M,~\text{for all}~ j\neq n_k\}.
\end{align}
Denote by $s_M(\alpha )$ the unique real solution of the equation
$$\sum_{k=1}^{M}\left(\frac{1}{2^\alpha 2^k}\right)^s=1.$$
\begin{remark}\label{Remark}
  By the definitions of $s(\alpha )$ and $s_M(\alpha )$, we have
  $$\lim_{M\to \infty}s_M(\alpha )=s(\alpha ).$$
\end{remark}
Now, we provide a symbolic description of the structure of $F_M(\alpha )$.  Let
\begin{align*}
  D_n=&\{  (\sigma_1,\cdots,\sigma_n)\in \mathbb{N}^n:\lfloor \alpha n_k\rfloor+1\le \sigma_{n_k}\le \lfloor(1+1/k)\alpha n_k\rfloor, \\
   &~\text{for all}~ k\ge 1, 1\le n_k\le n, ~\text{and}~1\le \sigma_j\le M,~\forall j\neq n_k, 1\le j\le n\},\\
  D=& \bigcup_{n=1}^{\infty}D_n.
\end{align*}
For any $n\ge 1$ and any $(\sigma_1,\cdots,\sigma_n)\in D_n$, we call $I_n(\sigma_1,\cdots,\sigma_n)$ a basic interval of rank $n$. Furthermore,
\begin{equation}\label{J}
  J(\sigma_1,\cdots,\sigma_n)=\bigcup_{\sigma_{n+1}: (\sigma_1,\cdots,\sigma_{n+1})\in D_{n+1}}I_{n+1}(\sigma_1,\cdots,\sigma_n, \sigma_{n+1})
\end{equation}
is termed a fundamental interval of rank $n$. According to (\ref{|I_n|}), if $n+1\neq n_k$ for any $k\ge 1$, then
\begin{equation}\label{|J(1)|}
  |J(\sigma_1,\cdots,\sigma_n)|=\frac{1}{2^{\sigma_1+\cdots+\sigma_n}}\left(1-\frac{1}{2^M}\right).
\end{equation}
If $n+1=n_k$ for some $k\ge 1$, then
\begin{equation}\label{|J(2)|}
  |J(\sigma_1,\cdots,\sigma_n)|=\frac{1}{2^{\sigma_1+\cdots+\sigma_n}}\left(\frac{1}{2^{\lfloor \alpha n_k\rfloor}}-\frac{1}{2^{\lfloor(1+1/k)\alpha n_k\rfloor}}\right).
\end{equation}
It is easily seen that
\begin{equation}\label{F_M(B)}
  F_M(\alpha )=\bigcap_{n\ge 1}\bigcup_{(\sigma_1,\cdots,\sigma_{n})\in D_n}J(\sigma_1,\cdots,\sigma_n).
\end{equation}

\subsubsection{A mass distribution on $F_M(\alpha )$}
In this subsection, we define a probability measure supported on $F_M(\alpha )$ and give an upper estimation of $\mu(J(\sigma_1,\cdots,\sigma_{n}))$ for each $(\sigma_1,\cdots,\sigma_{n})\in D_n$, by using the length $|J(\sigma_1,\cdots,\sigma_{n})|$.

\begin{proposition}\label{Lemma}
  Let $\epsilon>0$. There exists a probability measure $\mu$ supported on $F_M(\alpha )$ such that
  \begin{equation}\label{mu C}
    \mu(J(\sigma_1,\cdots,\sigma_{n}))\le C|J(\sigma_1,\cdots,\sigma_{n})|^{(1-\epsilon)s_M(\alpha )},
  \end{equation}
  where $C$ is a constant depending on $M$, $\epsilon$ and $\alpha $.
\end{proposition}

\begin{proof}

We define a set function $\mu:\{J(\sigma):\sigma\in D\}\to \mathbb{R}^{+}$ as follows.
For any $n<n_1$ and $(\sigma_1,\cdots, \sigma_n)\in D_n$, let
\begin{equation*}
  \mu(J(\sigma_1,\cdots, \sigma_n))=\prod_{i=1}^n\left(\frac{1}{2^\alpha 2^{\sigma_i}}\right)^{s_M(\alpha )}.
\end{equation*}
For any $(\sigma_1,\cdots,\sigma_{n_1})\in D_{n_1}$, set
$$\mu(J(\sigma_1,\cdots,\sigma_{n_1}))=\frac{1}{\lfloor2\alpha n_1\rfloor-\lfloor \alpha n_1\rfloor}\mu(J(\sigma_1,\cdots,\sigma_{n_1-1})).$$

 Once $\mu(J(\sigma_1,\cdots,\sigma_{n_k}))$ has been defined for some $k>1$ and any $(\sigma_1,\cdots,\sigma_{n_k})\in D_{n_k}$, we define 
$$\mu(J(\sigma_1,\cdots,\sigma_n))=\mu(J(\sigma_1,\cdots,\sigma_{n_k}))\prod_{i=n_{k}+1}^{n}\left(\frac{1}{2^{\alpha} 2^{\sigma_i}}\right)^{s_M(\alpha )},$$
 for any $n_{k}<n<n_{k+1}$ and any $(\sigma_1,\cdots,\sigma_n)\in D_n$,
and
$$\mu(J(\sigma_1,\cdots,\sigma_{n_{k+1}}))=\frac{1}{\lfloor(1+1/(k+1))\alpha n_{k+1}\rfloor-\lfloor \alpha n_{k+1}\rfloor}\mu(J(\sigma_1,\cdots,\sigma_{n_{k+1}-1})),$$
for any $(\sigma_1,\cdots,\sigma_{n_{k+1}})\in D_{n_{k+1}}$.

It can be verified that for any $n\ge 1$ and $(\sigma_1,\cdots,\sigma_n)\in D_n$, we have
$$\mu(J(\sigma_1,\cdots,\sigma_n))=\sum_{\sigma_{n+1}:(\sigma_1,\cdots,\sigma_{n+1})\in D_{n+1}}\mu(J(\sigma_1,\cdots,\sigma_n,\sigma_{n+1})),$$
and
$$\sum_{\sigma_1\in D_1}\mu(J(\sigma_1))=1.$$
Thus $\mu$ can be extended to a probability measure supported on $F_M(\alpha )$ due to Kolmogorov's extension theorem. We denote the  measure still by $\mu$.
 By the definition of $\mu$, for any $k\ge 1$, $ n_{k}\leq n < n_{k+1}$ and any $(\sigma_1,\cdots,\sigma_{n})\in D_{n}$, we have
\begin{align}\label{mu J} \mu(J(\sigma_1,\cdots,\sigma_{n}))=\prod_{j=1}^{k}\frac{1}{\lfloor(1+1/j)\alpha n_j\rfloor-\lfloor \alpha n_j\rfloor}\prod_{i\in \Delta_{n}}\left(\frac{1}{2^\alpha 2^{\sigma_i}}\right)^{s_M(\alpha )},
\end{align}
where $\Delta_n:=\{1\le i\le n:i\neq n_j ~\text{for any} ~ j\}$. 

Now, we proceed with the proof of Proposition \ref{Lemma} by distinguishing three cases.

\textbf{Case 1}. $n=n_k$ for some $k\ge 1$. In this case, by (\ref{|J(1)|}) and (\ref{mu J}), we have
\begin{align*}
  \mu(J(\sigma_1,\cdots,\sigma_{n})) & =\prod_{j=1}^{k}\frac{1}{\lfloor(1+1/j)\alpha n_j\rfloor-\lfloor \alpha n_j\rfloor}\prod_{i\in \Delta_{n}}\left(\frac{1}{2^\alpha 2^{\sigma_i}}\right)^{s_M(\alpha )} \\
   & = \prod_{j=1}^{k}\frac{1}{\lfloor(1+1/j)\alpha n_j\rfloor-\lfloor \alpha n_j\rfloor} 2^{-\left(\sum_{i\in \Delta_{n}}\sigma_i+\alpha (n-k)\right)s_M(\alpha )},
\end{align*}
and
$$|J(\sigma_1,\cdots,\sigma_{n})|=2^{-\sum_{i=1}^{n}\sigma_i}\left(1-\frac{1}{2^M}\right).$$
Note that
\begin{equation*}
  \lim_{k\to \infty}\frac{\sum_{j=1}^{k}\sigma_{n_j}}{n_k}= \alpha \qquad \text{and} \qquad
  \lim_{k\to \infty}\frac{\sum_{i=1}^{n}\sigma_{i}}{n_k}\ge 1+\alpha.
\end{equation*}
For any $\epsilon>0$, 
 there exists $k_0$ depending on $M$, $\alpha$ and $\epsilon$, such that for all $k>k_0$, we have
$$\sum_{i=1}^{n}\sigma_i(1-\epsilon)s_M(\alpha )\le \sum_{i\in \Delta_n}\sigma_i s_M(\alpha )+\alpha s_M(\alpha )(n-k).$$
Hence, there exists $c_0>0$ such that
\begin{equation}\label{mu(J1)}
  \mu(J(\sigma_1,\cdots,\sigma_{n}))\le c_0|J(\sigma_1,\cdots,\sigma_{n})|^{(1-\epsilon)s_M(\alpha )}.
\end{equation}

\textbf{Case 2}. $n+1=n_k$ for some $k\ge 1$. In this case, by (\ref{|J(2)|}) and (\ref{mu J}), we have
\begin{align*}
  \mu(J(\sigma_1,\cdots,\sigma_{n}))
   & = \prod_{j=1}^{k-1}\frac{1}{\lfloor(1+1/j)\alpha n_j\rfloor-\lfloor \alpha n_j\rfloor} 2^{-\left(\sum_{i\in \Delta_n}\sigma_i+\alpha \left(n-k+1\right)\right)s_M(\alpha )},\\
   |J(\sigma_1,\cdots,\sigma_{n})|&=2^{-\sum_{i=1}^{n}\sigma_i}\left(\frac{1}{2^{\lfloor \alpha n_k\rfloor}}-\frac{1}{2^{\lfloor(1+1/k)\alpha n_k\rfloor}}\right).
\end{align*}
Thus, there exists $c_1>0$ such that
$$c_1 2^{-\sum_{i=1}^{n}\sigma_i}\frac{1}{2^{\lfloor\alpha n_k\rfloor}}\le |J(\sigma_1,\cdots,\sigma_{n})|\le  2^{-\sum_{i=1}^{n}\sigma_i}\frac{1}{2^{\lfloor\alpha n_k\rfloor}}.$$
For any $\epsilon>0$, by choosing a sufficient large $k_0$, we have
$$\sum_{i\in \Delta_n}\sigma_i s_M(\alpha)+\alpha s_M(\alpha)(n-k+1)\ge \sum_{i=1}^{n}\sigma_i(1-\epsilon)s_M(\alpha)+\alpha n_k(1-\epsilon)s_M(\alpha),\qquad \forall k\ge k_0.$$
Then, there exists $c_2>0$ such that
\begin{equation}\label{mu(J2)}
  \mu(J(\sigma_1,\cdots,\sigma_{n}))\le c_2|J(\sigma_1,\cdots,\sigma_{n})|^{(1-\epsilon)s_M(\alpha)}.
\end{equation}

\textbf{Case 3}. $n_{k-1}<n< n_k-1$ for some $k>1$. By (\ref{|J(1)|}) and (\ref{mu J}), we have
\begin{align*}
  \mu(J(\sigma_1,\cdots,\sigma_{n})) & = \prod_{j=1}^{k-1}\frac{1}{\lfloor(1+1/j)\alpha n_j\rfloor-\lfloor \alpha n_j\rfloor} 2^{-\left(\sum_{i\in \Delta_n}\sigma_i+\alpha \left(n-k+1\right)\right)s_M(\alpha )}, \\
  |J(\sigma_1,\cdots,\sigma_{n})| & =2^{-\sum_{i=1}^{n}\sigma_i}\left(1-\frac{1}{2^M}\right).
\end{align*}
For any $\epsilon>0$, by choosing $k_0$ large enough, we have
$$\sum_{i=1}^{n}\sigma_i(1-\epsilon)s_M(\alpha )\le \sum_{i\in \Delta_n}\sigma_i s_M(\alpha )+\alpha s_M(\alpha )(n-k+1), \qquad \forall k\ge k_0.$$
Hence, there exists $c_3>0$ such that
\begin{equation}\label{mu(J3)}
  \mu(J(\sigma_1,\cdots,\sigma_{n}))\le c_3|J(\sigma_1,\cdots,\sigma_{n})|^{(1-\epsilon)s_M(\alpha )}.
\end{equation}
By taking $C=\max\{c_0, c_2, c_3\}$, we can deduce from (\ref{mu(J1)}), (\ref{mu(J2)}) and (\ref{mu(J3)}) that
$$\mu(J(\sigma_1,\cdots,\sigma_{n}))\le C|J(\sigma_1,\cdots,\sigma_{n})|^{(1-\epsilon)s_M(\alpha )}.$$
\end{proof}

\subsubsection{Estimation of gaps between fundamental intervals}
In this subsection, we focus on estimating the gaps between disjoint fundamental intervals, as defined in (\ref{J}), of the same rank.  For each $(\sigma_1,\cdots,\sigma_n)\in D_n$, we denote the distance between $J(\sigma_1,\cdots,\sigma_n)$ and the nearest fundamental interval of the same rank on its right (left, respectively) as $g^r(\sigma_1,\cdots,\sigma_n)$ ($g^{\ell}(\sigma_1,\cdots,\sigma_n)$, respectively).
We define the gap as $$g(\sigma_1,\cdots,\sigma_n):=\min\{g^r(\sigma_1,\cdots,\sigma_n), g^{\ell}(\sigma_1,\cdots,\sigma_n)\}.$$

\textbf{Case 1}: $n+1\neq n_k$ for any $k\ge 1$.
In this case, the left and right endpoints of $J(\sigma_1,\cdots,\sigma_n)$ are
$$\sum_{i=1}^{n}\frac{1}{2^{\sigma_1+\cdots+\sigma_i}}+\frac{1}{2^{\sigma_1+\cdots+\sigma_n+M}}\qquad\text{and}\qquad\sum_{i=1}^{n}\frac{1}{2^{\sigma_1+\cdots+\sigma_i}}+\frac{1}{2^{\sigma_1+\cdots+\sigma_n}}.$$
If there exists $1\le i_0\le n$ such that $(\sigma_1,\cdots,\sigma_{i_0}+1,\cdots,\sigma_n)\in D_n$ and $J(\sigma_1,\cdots,\sigma_{i_0}+1,\cdots,\sigma_n)$ is the nearest  fundamental interval of rank $n$ lying on the left of $J(\sigma_1,\cdots,\sigma_n)$, then $g^{\ell}(\sigma_1,\cdots,\sigma_n)$ is precisely the distance between the left endpoint of $J(\sigma_1,\cdots,\sigma_n)$ and the right endpoint of  $J(\sigma_1,\cdots,\sigma_{i_0}+1,\cdots,\sigma_n)$.
Thus,
\begin{align*}
  g^{\ell}(\sigma_1,\cdots,\sigma_n) & = \sum_{i=1}^{n}\frac{1}{2^{\sigma_1+\cdots+\sigma_i}}+\frac{1}{2^{\sigma_1+\cdots+\sigma_n+M}}
  -\left(\sum_{i=1}^{i_0-1}\frac{1}{2^{\sigma_1+\cdots+\sigma_i}}+\sum_{i=i_0}^{n}\frac{1}{2^{\sigma_1+\cdots+\sigma_i+1}}+\frac{1}{2^{\sigma_1+\cdots+\sigma_n+1}}\right)\\
   & = \sum_{i=i_0}^{n-1}\frac{1}{2^{\sigma_1+\cdots+\sigma_i+1}}+\frac{1}{2^{\sigma_1+\cdots+\sigma_n+M}}
   \ge \frac{1}{2^{\sigma_1+\cdots+\sigma_n+M}}.
\end{align*}

If there exists $1\le i_1\le n$ such that $(\sigma_1,\cdots,\sigma_{i_1}-1,\cdots,\sigma_n)\in D_n$ and $J(\sigma_1,\cdots,\sigma_{i_1}-1,\cdots,\sigma_n)$ is the nearest  fundamental interval of rank $n$ on the right of $J(\sigma_1,\cdots,\sigma_n)$, then $g^r(\sigma_1,\cdots,\sigma_n)$ is just the distance between the right endpoint of $J(\sigma_1,\cdots,\sigma_n)$ and the left endpoint of  $J(\sigma_1,\cdots,\sigma_{i_1}-1,\cdots,\sigma_n)$.
Thus,
\begin{align*}
  g^r(\sigma_1,\cdots,\sigma_n) & = \sum_{i=1}^{i_1-1}\frac{1}{2^{\sigma_1+\cdots+\sigma_i}}+\sum_{i=i_1}^{n}\frac{1}{2^{\sigma_1+\cdots+\sigma_i-1}}+\frac{1}{2^{\sigma_1+\cdots+\sigma_n+M-1}}
   -\left(\sum_{i=1}^{n}\frac{1}{2^{\sigma_1+\cdots+\sigma_i}}+\frac{1}{2^{\sigma_1+\cdots+\sigma_n}}\right)\\
  &=\sum_{i=i_1}^{n-1}\frac{1}{2^{\sigma_1+\cdots+\sigma_i}}+\frac{1}{2^{\sigma_1+\cdots+\sigma_n+M-1}}
   \ge\frac{1}{2^{\sigma_1+\cdots+\sigma_n+M-1}}.
\end{align*}
Hence,
\begin{equation}\label{G3}
  g(\sigma_1,\cdots,\sigma_n)\ge\frac{1}{2^{\sigma_1+\cdots+\sigma_n+M}}.
\end{equation}
\textbf{Case 2}: $n+1=n_k$ for some $k\ge 1$.

In this case, $g^{\ell}(\sigma_1,\cdots,\sigma_n)$ is larger than the distance between the left endpoint of $I(\sigma_1,\cdots,\sigma_n)$ and the left endpoint of $J(\sigma_1,\cdots,\sigma_n)$, $g^r(\sigma_1,\cdots,\sigma_n)$ is larger than the distance between the right endpoint of $I(\sigma_1,\cdots,\sigma_n)$ and the right endpoint of $J(\sigma_1,\cdots,\sigma_n)$. Thus
\begin{align*}
 g^{\ell}(\sigma_1,\cdots,\sigma_n) & \ge \sum_{i=1}^{n}\frac{1}{2^{\sigma_1+\cdots+\sigma_i}}+\frac{1}{2^{\sigma_1+\cdots +\sigma_n+\lfloor(1+1/k)\alpha n_k\rfloor}}-\sum_{i=1}^{n}\frac{1}{2^{\sigma_1+\cdots+\sigma_i}}\\
                     & = \frac{1}{2^{\sigma_1+\cdots+\sigma_n}}\cdot\frac{1}{2^{\lfloor(1+1/k)\alpha n_k\rfloor}},\\
 g^r(\sigma_1,\cdots,\sigma_n) & \ge \sum_{i=1}^{n-1}\frac{1}{2^{\sigma_1+\cdots+\sigma_i}}+\frac{1}{2^{\sigma_1+\cdots +\sigma_n-1}}
 -\left(\sum_{i=1}^{n}\frac{1}{2^{\sigma_1+\cdots+\sigma_i}}+\frac{1}{2^{\sigma_1+\cdots +\sigma_n+\lfloor \alpha n_k\rfloor}}\right)\\
                     & = \frac{1}{2^{\sigma_1+\cdots+\sigma_n}}\left(1-\frac{1}{2^{\lfloor \alpha n_k\rfloor}}\right).
\end{align*}
Hence
\begin{equation}\label{G4}
  g(\sigma_1,\cdots,\sigma_n)\ge \frac{1}{2^{\sigma_1+\cdots+\sigma_n}}\cdot\frac{1}{2^{\lfloor(1+1/k)\alpha n_k\rfloor}}.
\end{equation}

If $n+1\neq n_k$ for any $k\ge 1$, by (\ref{|J(1)|}) and (\ref{G3}), we have
$$g(\sigma_1,\cdots,\sigma_n)\ge \frac{1}{2^M-1}|J(\sigma_1,\cdots,\sigma_n)|.$$
If $n+1= n_k$ for some $k\ge 1$, by (\ref{|J(2)|}) and (\ref{G4}), we have
$$g(\sigma_1,\cdots,\sigma_n)\ge \frac{1}{2^{\lfloor(\alpha n_k )/k \rfloor+1}-1}|J(\sigma_1,\cdots,\sigma_n)|.$$

\subsubsection{Lower bound of $\dim_{\mathrm{H}}F_M(\alpha)$}
In this subsection, we apply the mass distribution principle \cite[Proposition 4.2]{falconer2004fractal}  to obtain the lower bound of $\dim_{\mathrm{H}}F_M(\alpha)$.
\begin{proposition}\label{lemma2}
  Let $F_M(\alpha)$ be defined as in (\ref{F_M(B)}). Then 
  $$\dim_{\mathrm{H}}F_M(\alpha)\ge  s_M(\alpha).$$
\end{proposition}

\begin{proof}

For  $n\ge 1$ and $(\sigma_1,\cdots,\sigma_n)\in D_n$,
\begin{enumerate}[$\bullet$]
  \item  if $n+1\neq n_k$ for any $k\ge 1$, set
  \begin{equation}\label{G11}
  g_1(\sigma_1,\cdots,\sigma_n):=\frac{1}{2^{\sigma_1+\cdots+\sigma_n}2^M}\ge \frac{1}{2^M-1}|J(\sigma_1,\cdots,\sigma_{n})|,
\end{equation}
  \item if $n+1= n_k$ for some $k\ge 1$, set
  \begin{equation}\label{G12}
  g_1(\sigma_1,\cdots,\sigma_n):=\frac{1}{2^{\sigma_1+\cdots+\sigma_n}2^{\lfloor(1+1/k)\alpha n_k\rfloor}}\ge \frac{1}{2^{\lfloor(\alpha n_k)/k\rfloor+1}-1}|J(\sigma_1,\cdots,\sigma_{n})|.
\end{equation}
\end{enumerate}
Define $r_0=\min\limits_{1\le j\le n_{k_0}}\min\limits_{(\sigma_1,\cdots,\sigma_j)\in D_j}g_1(\sigma_1,\cdots,\sigma_j)$. For a fixed $x\in F_M(\alpha )$ and $0<r<r_0$ and set $\sigma_i=d_i(x)$. Let $\epsilon>0$, and $k_0=k_0(\epsilon)$ is defined as in Proposition \ref{Lemma}.
Then $x\in J(\sigma_1,\cdots,\sigma_k)$ for all $k\ge 1$ and there exists some $n\ge n_{k_0}$, such that
$$g_1(\sigma_1,\cdots,\sigma_n,\sigma_{n+1})\le r <g_1(\sigma_1,\cdots,\sigma_n).$$
By the definition of $g_1(\sigma_1,\cdots,\sigma_n)$,  the ball $B(x,r)$ can intersect only one fundamental interval  of rank $n$, which is precisely  $J(\sigma_1,\cdots,\sigma_n)$. To continue the proof, we distinguish three cases. 
\begin{enumerate}
  \item[\textbf{Case 1}:] $n+1=n_k$ for some $k\ge k_0$. We further distinguish two subcases.
\begin{itemize}
  \item [\textbf{Case 1.1}]$r\le \frac{1}{2}|I(\sigma_1,\cdots,\sigma_{n_k})|$. In this case, the ball $B(x,r)$ can intersect at most three basic intervals of rank $n_k$, which are $I(\sigma_1,\cdots,\sigma_{n_k}-1), I(\sigma_1,\cdots,\sigma_{n_k}), I(\sigma_1,\cdots,\sigma_{n_k}+1)$ (if exist). Hence, we have
      $$\mu(B(x,r))\le 3\mu(J(\sigma_1,\cdots,\sigma_{n_k})).$$
      By Proposition \ref{Lemma} and (\ref{G11}), we have
      \begin{equation}\label{mu B1}
      \begin{aligned}
        \mu(B(x,r))&\le 3C|J(\sigma_1,\cdots,\sigma_{n_k})|^{(1-\epsilon)s_M(\alpha )}\le 3(2^M-1)Cg_1(\sigma_1,\cdots, \sigma_{n_k})^{(1-\epsilon)s_M(\alpha )}
        \\&\le 3(2^M-1)Cr^{(1-\epsilon)s_M(\alpha )}.
        \end{aligned}
      \end{equation}
  \item [\textbf{Case 1.2}]$r\ge \frac{1}{2}|I(\sigma_1,\cdots,\sigma_{n_k})|$. Since
      $$|I(\sigma_1,\cdots\sigma_{n_k})|=\frac{1}{2^{\sigma_1+\cdots+\sigma_{n_k-1}}}\cdot\frac{1}{2^{\sigma_{n_k}}}\ge \frac{1}{2^{\sigma_1+\cdots+\sigma_{n_k-1}}}\cdot\frac{1}{2^{\lfloor(1+1/k)\alpha n_k\rfloor}},$$
      there are at  most $2r2^{\sigma_1+\cdots+\sigma_{n_k-1}}2^{\lfloor(1+1/k)\alpha n_k\rfloor}$
      fundamental intervals of rank $n_k$ contained in $J\left(\sigma_1,\cdots,\right.$\\
      $\left.\sigma_{n_k-1}\right)$ and intersecting the ball $B(x,r)$. Then,  by (\ref{|J(2)|}) and Proposition \ref{Lemma}, we have
      \begin{align*}
        &\mu(B(x,r))\le  \mu(J(\sigma_1,\cdots,\sigma_{n_k-1}))\min\left\{1,2r2^{\sigma_1+\cdots+\sigma_{n_k-1}}2^{\lfloor(1+1/k)\alpha n_k\rfloor}
        \frac{1}{\lfloor(1+1/k)\alpha n_k\rfloor-\lfloor \alpha n_k\rfloor}\right\} \\
         & \le C|J(\sigma_1,\cdots,\sigma_{n_k-1})|^{(1-\epsilon)s_M(\alpha )}\min\left\{1,2r2^{\sigma_1+\cdots+\sigma_{n_k-1}}
         2^{\lfloor(1+1/k)\alpha n_k\rfloor}\frac{1}{\lfloor(1+1/k)\alpha n_k\rfloor-\lfloor \alpha n_k\rfloor}\right\} \\
         & \le C|J(\sigma_1,\cdots,\sigma_{n_k-1})|^{(1-\epsilon)s_M(\alpha )}\left(2r2^{\sigma_1+\cdots+\sigma_{n_k-1}}2^{\lfloor(1+1/k)\alpha n_k\rfloor}
         \frac{1}{\lfloor(1+1/k)\alpha n_k\rfloor-\lfloor \alpha n_k\rfloor}\right)^{(1-2\epsilon)s_M(\alpha )} \\
         & \leq C\left(\frac{2r}{\lfloor(1+1/k)\alpha n_k\rfloor-\lfloor \alpha n_k\rfloor}\right)^{(1-2\epsilon)s_M(\alpha )}2^{-\sum_{i=1}^{n_k-1} \epsilon s_M(\alpha )\sigma_i}.
      \end{align*}
      Thus, there exists $c_4>0$ such that
      \begin{equation}\label{mu B2}
        \mu(B(x,r))\le c_4 r^{(1-2\epsilon)s_M(\alpha )}.
      \end{equation}

\end{itemize}

\item[\textbf{Case 2}:] $n+1 \neq n_k$ and $n+2 \neq n_k$ for any $k\ge k_0$.

For any $1\le \xi<\zeta\le M$, by the definition of $\mu$, we have
$$\frac{\mu(J(\sigma_1,\cdots,\sigma_n,\xi))}{\mu(J(\sigma_1,\cdots,\sigma_n,\zeta))} \le 2^{M-1}.$$
From Proposition \ref{Lemma} and (\ref{G11}), we deduce that
\begin{align}\label{mu B3}
  \mu(B(x,r)) &\nonumber\le M2^{M-1}\mu(J(\sigma_1,\cdots,\sigma_{n+1})) \\
   & \nonumber\le CM2^{M-1}|J(\sigma_1,\cdots,\sigma_{n+1})|^{(1-\epsilon)s_M(\alpha )} \\
   & \nonumber\le CM2^{2M-1}g_1(\sigma_1,\cdots,\sigma_{n+1})^{(1-\epsilon)s_M(\alpha )}\\
   & \le CM2^{2M-1}r^{(1-\epsilon)s_M(\alpha )}.
\end{align}

\item[\textbf{Case 3}:] $n+2 = n_k$ for any $k\ge k_0$.

Similar to Case 2,
we deduce from
Proposition \ref{Lemma} and (\ref{G12}) that
\begin{align}\label{mu B4}
  \mu(B(x,r)) &\nonumber\le M2^{M-1}\mu(J(\sigma_1,\cdots,\sigma_{n+1})) \\
   & \nonumber\le CM2^{M-1}|J(\sigma_1,\cdots,\sigma_{n+1})|^{(1-\epsilon)s_M(\alpha )} \\
   & \nonumber\le CM2^{M-1}(2^{\lfloor(\alpha n_k)/k \rfloor+1}-1)^{(1-\epsilon)s_M(\alpha )}
   g_1(\sigma_1,\cdots,\sigma_{n+1})^{(1-\epsilon)s_M(\alpha )}\\
   & \nonumber\le CM2^{M-1}g_1(\sigma_1,\cdots,\sigma_{n+1})^{(1-2\epsilon)s_M(\alpha )}\\
   &\le CM2^{M-1}r^{(1-2\epsilon)s_M(\alpha )}.
\end{align}

\end{enumerate}
 As a consequence of mass distribution principle \cite[Proposition 4.2]{falconer2004fractal} and (\ref{mu B1}) -- (\ref{mu B4}), we deduce that
$$\dim_{\mathrm{H}}F_M(\alpha )\ge (1-2\epsilon) s_M(\alpha ).$$
Taking $\epsilon \to 0$, we obtain
$$\dim_{\mathrm{H}}F_M(\alpha )\ge s_M(\alpha ).$$

\end{proof}
Since $F_M(\alpha )\subset F(\alpha )$, we have
$$\dim_{\mathrm{H}}F(\alpha )\ge \dim_{\mathrm{H}}F_M(\alpha )\ge s_M(\alpha ).$$
Letting $M\to \infty$, by Remark \ref{Remark}, we obtain
$$\dim_{\mathrm{H}}F(\alpha )\ge s(\alpha ).$$

\begin{corollary}\label{cor4.1}
  For any infinite set $\mathcal{N}\subset \mathbb{N}$,
  $$\dim_{\mathrm{H}}\{x\in (0,1]: d_n(x)\ge \alpha n ~~\text{i.m.} ~ n \in \mathcal{N}\}=s(\alpha).$$
\end{corollary}
\begin{proof}
  It follows by choosing the integer sequence $\{n_k\}_{k\ge 1}$ satisfying (\ref{n_k}) and $n_k\in \mathcal{N}$ for all $k\ge 2$ in the proof of Theorem \ref{addthm}.
\end{proof}

\subsection{Hausdorff dimension of $F(\phi)$}
\begin{proof}[Proof of Theorem \ref{main theorem im}]

  \begin{enumerate}
    \item [(1)]If $\alpha=0$, then for any $\varepsilon>0$, $\phi(n)\le \varepsilon n$ holds for infinitely many $n$'s. Let
        $$\mathcal{A}=\{n:\phi(n)\le \varepsilon n\}.$$
        Then
        $$\{x\in (0,1]:d_n(x)\ge \varepsilon n,~~\text{i.m.}~n\in \mathcal{A}\}\subset F(\phi).$$
        By Lemma \ref{Lemma s(B)} and Corollary \ref{cor4.1}, we have
        $$\dim_{\mathrm{H}}F(\phi)\ge \sup_{\varepsilon>0}s(\varepsilon)=1,$$
        thus $\dim_{\mathrm{H}}F(\phi)=1$.
    \item [(2)]If $0<\alpha<\infty$, then for any $\varepsilon>0$, $\phi(n)\le (\alpha+\varepsilon) n$ holds for infinitely many $n$'s. On the other hand, there exists $N\in \mathbb{N}$ such that for all $n\ge N$, we have $(\alpha-\varepsilon)n\le \phi(n)$. Let
        $$\mathcal{A}=\{n:\phi(n)\le (\alpha+\varepsilon) n\}.$$
        Then,
        $$\{x\in (0,1]:d_n(x)\ge (\alpha+\varepsilon) n,~~\text{i.m.}~n\in \mathcal{A}\}\subset F(\phi)\subset \{x\in I:d_n(x)\ge (\alpha-\varepsilon) n,~~\text{i.m.}~n\}.$$
        By Lemma \ref{Lemma s(B)} and Corollary \ref{cor4.1} again, we have
        $$s(\alpha)=\lim_{\varepsilon\to 0}s(\alpha+\varepsilon)\le \dim_{\mathrm{H}}F(\phi)\le \lim_{\varepsilon\to 0}s(\alpha-\varepsilon)=s(\alpha).$$
    \item [(3)]If $\alpha=\infty$, then for any $\alpha_1\ge 0$, $\phi(n)\ge \alpha_1n$ holds ultimately. Then
    $$F(\phi)\subset \{x\in (0,1] :d_n(x)\ge \alpha_1n,~~\text{i.m.}~n\}.$$
    Thus,
    $$ \dim_{\mathrm{H}}F(\phi)\le s(\alpha_1).$$
    Letting $\alpha_1\to \infty$ gives
    $$\dim_{\mathrm{H}}F(\phi)=0.$$
  \end{enumerate}
\end{proof}

\section{Proof of Theorem \ref{main theorem 2}}
For any $r>0$ and $\alpha >0$, recall
$$L_n(x)=\max\{d_i(x):1\le i\le n\},$$
and
$$E(r,\alpha)=  \left\{x\in (0,1]:\lim_{n\to \infty}\frac{L_n(x)}{n^r}=\alpha\right\}.$$
The proof falls naturally into three parts: $r>1$, $0<r<1$ and $r=1$.

\subsection{The case $r>1$}
Note that for any $x\in E(r,\alpha)$ and $\epsilon>0$, we have 
$$L_n(x)\ge (1-\epsilon)\alpha n^r~~\text{i.m.}~n.$$
By Lemma \ref{use3}, we obtain
$$d_n(x)\ge (1-\epsilon)\alpha n^r~~\text{i.m.}~n.$$
Therefore,
$$E(r,\alpha)\subset \{x\in (0,1]: d_n(x)\ge (1-\epsilon)\alpha n^r~~\text{i.m.}~n\}.$$
By Theorem \ref{main theorem im}, we have
$$\dim_{\mathrm{H}}E(r,\alpha)=0.$$

\subsection{The case $0<r<1$}
  In this case, it is evident that $\dim_{\mathrm{H}}E(r,\alpha)\le 1$. To establish the lower bound of $\dim_{\mathrm{H}}E(r,\alpha)$, we will construct a subset $E_M(r,\alpha)$ of $E(r,\alpha)$ and demonstrate that the Hausdorff dimension of $E_M(r,\alpha)$ can approximate $1$.

Fix an integer $M\ge 2$. Let $t>0$ be an integer such that $r+\frac{1}{t}<1$. For each $k\ge 1$, let $n_k=k^t$. Write
\begin{align*}
  E_M(r,\alpha)=\{ & x\in (0,1]: d_{n_k}(x)=\lfloor\alpha n_k^r\rfloor ~\text{for all}~k\ge 1~\text{and}~
    1\le d_i(x)\le M~\text{for all}~i\neq n_k\}.
\end{align*}

\begin{proposition}\label{subset}
  Let $0<r<1$ and $\alpha>0$. Then $ E_M(r,\alpha)\subset  E(r,\alpha)$.
\end{proposition}
\begin{proof}
  Fix $x\in  E_M(r,\alpha)$. Choose $k_0$ large enough such that $\lfloor\alpha n_{k_0}^r\rfloor\ge M$. For any $n> n_{k_0}$, let $k>k_0$ be such that $n_k\le n <n_{k+1}$.  Consequently,
  $$L_n(x)=\max\{d_1(x),\cdots,d_n(x)\}=\lfloor\alpha n_k^r\rfloor.$$
  Then,
  $$\frac{\lfloor\alpha n_k^r\rfloor}{n_{k+1}^r}<\frac{L_n(x)}{n^r}\le \frac{\lfloor\alpha n_k^r\rfloor}{n_{k}^r}.$$
  Noting that
  $$\lim_{k\to \infty}\frac{n_{k+1}^r}{n_{k}^r}=1,$$
  we obtain
  $$\lim_{n\to \infty}\frac{L_n(x)}{n^r}=\alpha.$$
  This confirms that $E_M(r,\alpha)\subset  E(r,\alpha)$.
\end{proof}
Now, we provide a symbolic description of the structure of $E_M(r,\alpha)$. Let
  \begin{align*}
    D_n= & \{(\sigma_1,\cdots,\sigma_n)\in \mathbb{N}^n:\sigma_{n_k}(x)=\lfloor \alpha n_k^r\rfloor,~\text{for all}~k\ge 1, 1\le n_k\le n,\\
    &~\text{and}~1\le \sigma_i(x)\le M ~\text{for all}~i\neq n_k, 1\le i\le n\}, \\
    D= & \bigcup_{n=1}^{\infty}D_n.
  \end{align*}
  For any $n\ge 1$ and any $(\sigma_1,\cdots,\sigma_n)\in D_n$, recall that we denote $I_n(\sigma_1,\cdots,\sigma_n)$ as a basic interval of rank $n$, and
  $$J_n(\sigma_1,\cdots,\sigma_n)=\bigcup_{\sigma_{n+1}:(\sigma_1,\cdots, \sigma_{n+1})\in D_{n+1}}I_{n+1}(\sigma_1,\cdots,\sigma_n,\sigma_{n+1})$$
  as a fundamental interval of rank $n$. It can be verified that
  \begin{equation*}\label{E_M1}
    E_M(r,\alpha)=\bigcap_{n\ge 1}\bigcup_{(\sigma_1,\cdots,\sigma_n)\in D_n}I_n(\sigma_1,\cdots,\sigma_n).
  \end{equation*}

  For any $n\ge 1$, let $\tau(n)=\#\{k: n_k\le n\}$. Since $n_k=k^t$, we have $\tau(n)\le n^{\frac{1}{t}}$ for large $n$. For any $(\sigma_1,\cdots,\sigma_n)\in D_n$, let $(\overline{\sigma_1,\cdots,\sigma_n})$ be the block obtained by eliminating the terms $\{\sigma_{n_k}:1\le k\le \tau(n)\}$ in $(\sigma_1,\cdots,\sigma_n)$.
  Thus $(\overline{\sigma_1,\cdots,\sigma_n})$  is of length $n-\tau(n)$ and $(\overline{\sigma_1,\cdots,\sigma_n})\in \mathcal{A}^{n-\tau(n)}$, where $\mathcal{A}=\{1,2,\cdots,M\}$.
   Set
  $$\overline{I_n}(\sigma_1,\cdots,\sigma_n)=I_{n-\tau(n)}(\overline{\sigma_1,\cdots,\sigma_n}).$$

   Fix $\alpha>0$. For any $x=(\sigma_1, \sigma_2,\cdots)\in E_M(r,\alpha)$, define a map $f:E_M(r,\alpha)\to E_M$ as
 $$f(x)=\tilde{x}=(\overline{\sigma_1,\sigma_2,\cdots}),$$
 where $E_M$ is defined as in (\ref{EM}) and $(\overline{\sigma_1,\sigma_2,\cdots})$ is the block obtained by eliminating the terms $\{\sigma_{n_k}:k\ge 1\}$ in $(\sigma_1,\sigma_n,\cdots)$. We will show that $f$ is almost Lipschitz on the set $E_M(r,\alpha)$.

 \begin{proposition}\label{Lipsichtz}
   For any $\epsilon>0$, $f$ is $1/(1+\epsilon)$-H\"{o}lder, i.e., for any $x,y\in E_M(r,\alpha)$,
   $$|f(x)-f(y)|\le 2^{\frac{2M}{1+\epsilon}}|x-y|^{\frac{1}{1+\epsilon}},$$
   if $x$ and $y$ are sufficiently close to each other.
 \end{proposition}
    
  \begin{proof}
    The proof will be divided into three steps.
    
    \textbf{Step 1:} Compare the lengths between $\overline{I_n}(\sigma_1,\cdots,\sigma_n)$ and $I_n(\sigma_1,\cdots,\sigma_n)$.
    
    For any $\epsilon>0$, since $r+\frac{1}{t}<1$, there exists $N_0=N_0(\epsilon)$ such that, for any $n\ge N_0$ and any $(\sigma_1,\cdots,\sigma_n)\in D_n$, we have
    \begin{equation*}
       |\overline{I_n}(\sigma_1,\cdots,\sigma_n)|^{\epsilon}
       =\left(\prod_{i=1\atop i\neq n_k, k\ge 1}^n\frac{1}{2^{\sigma_i}}\right)^{\epsilon}\le \frac{1}{2^{(n-\tau(n))\epsilon}}\le \frac{1}{2^{\alpha n^r \tau(n)}}.
    \end{equation*}
    Hence, for any $n\ge N_0$
    \begin{align}\label{compare}
      |{I_n}(\sigma_1,\cdots,\sigma_n)| &\nonumber =|\overline{I_n}(\sigma_1,\cdots,\sigma_n)|\prod_{i=1}^{\tau(n)}\frac{1}{2^{\sigma_{n_i}}}\ge |\overline{I_n}(\sigma_1,\cdots,\sigma_n)|2^{-\alpha n^r \tau(n)} \\
       & \ge |\overline{I_n}(\sigma_1,\cdots,\sigma_n)|^{1+\epsilon}.
    \end{align}
    
    \textbf{Step 2:}
    Estimate the distance between $x, y\in E_M(r,\alpha), x\neq y.$
     
     Let $x, y \in E_M(r,\alpha)$ with $x\neq y$. We can identify the greatest integer, denoted by $n_0$, for which $x, y$ are both contained in the same basic interval of rank $n_0$ $I_{n_0}(\sigma_1,\cdots,\sigma_{n_0})$, but not in the same basic interval of rank $n_0+1$. In other words, there exist $(\sigma_1,\sigma_2,\cdots,\sigma_{n_0})\in D_{n_0}$ and $\ell\neq r\in \mathbb{N}$ such that $(\sigma_1,\cdots,\sigma_{n_0},\ell)\in D_{n_0+1}$, $(\sigma_1,\cdots,\sigma_{n_0},r)\in D_{n_0+1}$ and $x\in J_{n_0+1}(\sigma_1,\cdots,\sigma_{n_0},\ell)$, $y\in J_{n_0+1}(\sigma_1,\cdots,\sigma_{n_0},r)$ respectively.

 Due to the maximality of $n_0$, we have ${n_0}+1\neq n_k$ for any $k\in \mathbb{N}$, otherwise $\sigma_{{n_0}+1}=\lfloor\alpha n_k\rfloor$ would hold for both $x$ and $y$ for some $k\ge 1$, contradicting the maximality of $n_0$. As a consequence,  the distance $|y-x|$  is at least the gap between $J_{{n_0}+1}(\sigma_1, \cdots, \sigma_{n_0}, \ell)$ and $J_{{n_0}+1}(\sigma_1, \cdots, \sigma_{n_0}, r)$. 

   Without loss of generality, we assume $x<y$.
   Then $1\le r<\ell\le m$. We distinguish two cases.
   \begin{itemize}
   \item [\textbf{Case 1}:]If $n+2=n_k$ for some $k\ge 1$, then $y-x$ exceeds the distance between the right endpoint of $J_{{n_0}+1}(\sigma_1,\cdots,\sigma_{n_0}, \ell)$ and the left endpoint of $J_{{n_0}+1}(\sigma_1,\cdots,\sigma_{n_0}, r)$. Thus by Lemma \ref{Lemma endpoint}, we have
         \begin{align*}
           |y-x| & \ge
           \left(\sum_{i=1}^{n}\frac{1}{2^{\sigma_1+\cdots+\sigma_i} } +\frac{1}{2^{\sigma_1+\cdots+\sigma_{n_0}+r}} +\frac{1}{2^{\sigma_1+\cdots+\sigma_{n_0}+r+\sigma_{n+2}}}\right) \\
            & - \left(\sum_{i=1}^{n}\frac{1}{2^{\sigma_1+\cdots+\sigma_i} } +\frac{1}{2^{\sigma_1+\cdots+\sigma_{n_0}+\ell}} +\frac{1}{2^{\sigma_1+\cdots+\sigma_{n_0}+\ell+\sigma_{n+2}-1}}\right)\\
            & = |I_{n_0}|\left(\frac{1}{2^{r}}-\frac{1}{2^{\ell}}
            +\frac{1}{2^{r+\sigma_{n+2}}}-\frac{1}{2^{\ell+\sigma_{n+2}-1}}\right) \\
            & \ge |I_{n_0}|\left(\frac{1}{2^r}-\frac{1}{2^{\ell}}\right)
             \ge \frac{|I_{n_0}|}{2^M}.
         \end{align*}

         \item [\textbf{Case 2}:]If $n+2\neq n_k$ for any $k\ge 1$,
         then $y-x$ is greater than the distance between the right endpoint of $J_{{n_0}+1}(\sigma_1,\cdots,\sigma_{n_0}, \ell)$ and the left endpoint of $J_{{n_0}+1}(\sigma_1,\cdots,\sigma_{n_0}, r)$. Thus
         \begin{align*}
           |y-x| & \ge \left(\sum_{i=1}^{n}\frac{1}{2^{\sigma_1+\cdots+\sigma_i} } +\frac{1}{2^{\sigma_1+\cdots+\sigma_{n_0}+r}} +\frac{1}{2^{\sigma_1+\cdots+\sigma_{n_0}+r+M}}\right) \\
            & -\left(\sum_{i=1}^{n}\frac{1}{2^{\sigma_1+\cdots+\sigma_i} } +\frac{1}{2^{\sigma_1+\cdots+\sigma_{n_0}+\ell}} +\frac{1}{2^{\sigma_1+\cdots+\sigma_{n_0}+\ell}}\right) \\
            & = |I_{n_0}|\left(\frac{1}{2^{r}}-\frac{1}{2^{\ell-1}}
            +\frac{1}{2^{r+M}}\right) \ge \frac{|I_{n_0}|}{2^{2M}}.
         \end{align*}
   \end{itemize}
    Then,
     \begin{equation}\label{min}
       |y-x|\ge \frac{|I_{n_0}(\sigma_1,\cdots,\sigma_{n_0})|}{2^{2M}}.
     \end{equation}

 \textbf{Step 3:} Let
 $$\delta:=
    \frac{|I_{N_0}(\sigma_1,\cdots,\sigma_{N_0})|}{2^{2M}}>0,$$
 where $N_0$ is defined as in \textbf{Step 1}. Then it follows from (\ref{compare}) and (\ref{min}) that, when $x,y\in E_M(r,\alpha)$ with $|x-y|<\delta$, we have
 $$|f(x)-f(y)|\le 2^{\frac{2M}{1+\epsilon}}|x-y|^{\frac{1}{1+\epsilon}}.$$
  \end{proof}   
    
 Thus, by Lemma \ref{lemmasM1}, Proposition \ref{subset}, Proposition \ref{Lipsichtz} and \cite[Proposition 3.3]{falconer2004fractal}, we have
 $$\dim_{\mathrm{H}}E(r,\alpha)\ge \dim_{\mathrm{H}}E_M(r,\alpha)\ge \frac{1}{1+\epsilon}\dim_{\mathrm{H}}E_M= \frac{1}{1+\epsilon} \log_2(s_M).$$
 Letting $\epsilon\to 0$ and $M\to \infty$, we deduce from Lemma \ref{lemmasM1} that
 $$\dim_{\mathrm{H}}E(r,\alpha)=1.$$   
 
 \begin{proof}[Proof of Corollary \ref{coro1.1}]
   From the above proofs, if we replace $d_{n_k}=\lfloor\alpha  n_k^r\rfloor$ by $d_{n_k}=\lfloor\alpha \log n_k\rfloor$,  (\ref{compare}) and (\ref{min}) remain valid. Then, 
   $$\dim_{\mathrm{H}}\Big\{x\in (0,1]:\lim_{n\to \infty}\frac{L_n(x)}{\log_2 n}=\alpha\Big\}=1.$$
 \end{proof}

 \subsection{The case $r=1$}

 In this subsection, we set
 $$E(\alpha):=E(1,\alpha)=\left\{x\in (0,1]:\lim_{n\to \infty}\frac{L_n(x)}{n}=\alpha\right\}.$$
We have
 \begin{equation}\label{cover 1}
   E(\alpha)=\bigcap_{M=1}^{\infty}\bigcup_{N=1}^{\infty}\bigcap_{n\ge N}\left\{x\in (0,1]: \left(1-\frac{1}{2M+1}\right)\alpha n\le L_n(x)\le \left(1+\frac{1}{2M+1}\right)\alpha n\right\}.
 \end{equation}
  For fixed $N\ge 1$ and  $M>1$,  we have
 \begin{align}\label{cover 2}
    &\nonumber\bigcap\limits_{n\ge N} \left\{x\in (0,1]: \left(1-\frac{1}{2M+1}\right)\alpha n\le L_n(x)\le \left(1+\frac{1}{2M+1}\right)\alpha n\right\} \\
   &\subset  \bigcap\limits_{k=1}^{M}\left\{x\in (0,1]: \left(1-\frac{1}{2M+1}\right)\alpha kN\le L_{kN}(x)\le \left(1+\frac{1}{2M+1}\right)\alpha kN\right\}.
 \end{align}
For any $1\le k\le M$, set
 $$\Theta_N(k)=\{i_k\in \mathbb{N}: (k-1)N+1\le i_k \le kN\}.$$
 
  Note that 
 \[\left(1+\frac{1}{2M+1}\right)(k-1)< \left(1-\frac{1}{2M+1}\right)k,\qquad \forall 1\le k\le M.\]
 Then for any $x$ in right-hand side of (\ref{cover 2}),  we have
  $$L_{kN}(x)= \max\limits_{i_k\in\Theta_N(k)}\{ d_{i_k}(x)\}.$$
 Thus, there exists $(i_1,i_2,\cdots,i_M)\in \prod_{k=1}^{M}\Theta_N(k) $ such that $L_{kN}(x)=d_{i_k}(x)$ for any $1\le k\le M$. Write
 \begin{align*}
   &\Delta_{MN}(i_1,i_2,\cdots,i_M)=\Big\{ (d_1,d_2,\ldots,d_{MN})\in \mathbb{N}^{MN}: d_{j}\geq 1~\text{for all}~j\neq i_k,~\text{and}\\
    &\left(1-\frac{1}{2M+1}\right)\alpha kN\leq d_{i_k}\leq \left(1-\frac{1}{2M+1}\right)\alpha kN~~ \text{for any}~i_k\in\Theta_N(k), 1\le k\le M\Big\}. 
 \end{align*}
  Therefore,
  
 \begin{align}\label{cover 3}
   &\nonumber  \bigcap\limits_{k=1}^{M}\left\{x\in (0,1]: \left(1-\frac{1}{2M+1}\right)\alpha kN\le L_{kN}(x)\le \left(1+\frac{1}{2M+1}\right)\alpha kN\right\} \\
    \subset & \bigcup_{(i_1,i_2,\cdots,i_M)\in \prod_{k=1}^{M}\Theta_N(k)}\bigcup\limits_{(d_1,d_2,\ldots,d_{MN})\in \Delta_{MN}(i_1,i_2,\cdots,i_M)} I_{MN}(d_1,d_2,\cdots,d_{MN}).
 \end{align}
 From (\ref{cover 1}), (\ref{cover 2}) and (\ref{cover 3}), we deduce that the set $E(\alpha)$ has natural coverings:
 $$E(\alpha)\subset \bigcup_{N=1}^{\infty}\bigcup_{(i_1,i_2,\cdots,i_M)\in \prod_{k=1}^{M}\Theta_N(k)}\bigcup\limits_{(d_1,d_2,\ldots,d_{MN})\in \Delta_{MN}(i_1,i_2,\cdots,i_M)} I_{MN}(d_1,d_2,\cdots,d_{MN}),\qquad\forall M\ge 1.$$

For fixed $(i_1,i_2,\cdots,i_M)\in \prod_{k=1}^{M}\Theta_N(k)$ and $0<s<1$, we  have
 \begin{align*}
&\sum\limits_{(d_1,d_2,\ldots,d_{MN})\in \Delta_{MN}(i_1,i_2,\cdots,i_M)} | I_{MN}(d_1,d_2,\cdots,d_{MN})|^s\\&=
\sum\limits_{(d_1,d_2,\ldots,d_{MN})\in \Delta_{MN}(i_1,i_2,\cdots,i_M)} \frac{1}{2^{(d_1(x)+d_2(x)+\cdots+d_{MN}(x))s}}\\&=
 \prod_{j\neq i_k}\left(\sum_{d_j(x)=1}^{\infty}\frac{1}{2^{d_j(x)}}\right)^s\prod_{k=1}^{M}
    \left(\sum_{d_{i_k}(x)=\lfloor(1-\frac{1}{2M+1})k\alpha N\rfloor}^{\lfloor(1+\frac{1}{2M+1})k\alpha N\rfloor}\frac{1}{2^{d_{i_k}(x)}}\right)^s\\
    &\le\left(\frac{1}{2^s-1}\right)^{MN-M} \prod_{k=1}^{M}\frac{4}{(2^s-1)2^{(1-\frac{1}{2M+1})k\alpha Ns}},
 \end{align*}
which is independent of the choice of $(i_1,i_2,\cdots,i_M)$. Since $\# \left(\prod_{k=1}^{M}\Theta_N(k)\right)=N^M$, we have
 \begin{align*}
   \mathcal{H}^s(E(\alpha))
    &\le \liminf_{M\to \infty}\sum_{N=1}^{\infty}
    (4N)^M\left(\frac{1}{2^s-1}\right)^{MN-M} \prod_{k=1}^{M}\frac{}{(2^s-1)2^{(1-\frac{1}{2M+1})k\alpha Ns}}\\
    &=\liminf_{M\to \infty}\sum_{N=1}^{\infty}(4N)^M \Bigg(\frac{1}{(2^s-1)2^{(1-\frac{1}{2M+1})\frac{\alpha(M+1)s}{2}}}\Bigg)^{MN}.
 \end{align*}
 For any $s>0$ and any $\alpha>0$, let $a(M)=(2^s-1)2^{(1-\frac{1}{2M+1})\frac{\alpha(M+1)s}{2}}$. It follows that
 $$\lim\limits_{M\to+\infty}a(M)=+\infty.$$ 
 Then
 \begin{align*}
   \mathcal{H}^s(E(\alpha)) & \le \liminf_{M\to \infty}\sum_{N=1}^{\infty}(4N)^M\Bigg(\frac{1}{a(M)}\Bigg)^{MN}
     \le \liminf_{M\to \infty}\sum_{N=1}^{\infty}\Bigg(\frac{100}{a(M)}\Bigg)^{MN}
    \leq \liminf_{M\to \infty}\Bigg(\frac{100}{a(M)}\Bigg)^{M}  =0,
 \end{align*}
 where the second inequality comes form the fact $N^M\le 100^{MN}$. Therefore,
 $$\dim_{\mathrm{H}}E(\alpha)=0.$$

 \section{Proof of Theorem \ref{main theorem 3}}
 If $r<1$, for any $\alpha>0$, we have
 $$\left\{x\in (0,1]:\lim_{n\to \infty}\frac{L_n(x)}{n^r}=\alpha\right\}\subset
 \left\{x\in (0,1]:\limsup_{n\to \infty}\frac{L_n(x)}{n^r}=\alpha\right\}.$$
 From Theorem \ref{main theorem 2}, we deduce that
 $$\dim_{\mathrm{H}}\left\{x\in (0,1]:\limsup_{n\to \infty}\frac{L_n(x)}{n^r}=\alpha\right\}=1.$$
  
 Now, we assume $r\ge 1$. For simplicity, we give only the proof for the case $r=1$. The case $r>1$ follows with similar arguments.
 
For any $\alpha>0$, set
$$\overline{E}(\alpha)=\left\{x\in (0,1]:\limsup_{n\to \infty}\frac{L_n(x)}{n}=\alpha\right\}.$$

\subsection{Upper bound of $\dim_{\mathrm{H}}\overline{E}(\alpha)$}
For any $\epsilon>0$, we have
$$
  \overline{E}(\alpha)
    \subset \bigcap_{\epsilon>0} \bigcap_{N=1}^{\infty}\bigcup_{n\ge N}\{L_n(x)\ge (1-\epsilon)\alpha n\}.
$$
By Lemma \ref{use3}, we obtain
$$\bigcap\limits_{N=1}^{\infty}\bigcup\limits_{n\ge N}\{L_n(x)\ge (1-\epsilon)\alpha n\}\subset \{x\in (0,1]: d_n(x)\ge (1-\epsilon)\alpha n~~\text{i.m.}~n\}.$$
Thus,
$$\overline{E}(\alpha)\subset \{x\in (0,1]: d_n(x)\ge (1-\epsilon)\alpha n^r~~\text{i.m.}~n\}.$$
From Theorem \ref{main theorem im}, we deduce that
$$\dim_{\mathrm{H}}\overline{E}(\alpha)\le s((1-\epsilon)\alpha).$$
Letting $\epsilon\to 0$, we obtain
$$\dim_{\mathrm{H}}\overline{E}(\alpha)\le s(\alpha).$$

\subsection{Lower bound of $\dim_{\mathrm{H}}\overline{E}(\alpha)$}

For $\alpha>0$,  we define a sequence $\{n_k\}_{k\in \mathbb{N}}$ as in (\ref{n_k}). 
Let $M\in \mathbb{N}$ and $F_M(\alpha)$ be defined as in (\ref{same}). Recall
\begin{align*}
  F_{M}(\alpha )=\{ &\nonumber x\in (0,1]:\lfloor \alpha n_k\rfloor+1\le d_{n_k}(x)\le \lfloor(1+1/k)\alpha n_k\rfloor,~\text{for all}~ k\ge 1, \\
   & \text{and}~1\le d_j(x)\le M,~\text{for all}~ j\neq n_k\}.
\end{align*}

\begin{proposition}
Let $\alpha>0$. Then $F_M(\alpha)\subset \overline{E}(\alpha)$.
\end{proposition}
\begin{proof}
Fix $x\in  F_M(\alpha)$,  we have
 $$\alpha\leq \limsup_{n\to \infty}\frac{L_n(x)}{n}= \limsup_{k\to \infty}\frac{d_{n_{k}}(x)}{n_{k}}\leq\alpha \limsup_{k\to \infty}\left(1+\frac{1}{k}\right)=\alpha.$$
 Then we obtain
 $$F_M(\alpha)\subset \overline{E}(\alpha).$$
\end{proof}

From Proposition \ref{lemma2}, we deduce that
$$\dim_{\mathrm{H}}\overline{E}(\alpha )\ge \dim_{\mathrm{H}}F_M(\alpha) \ge s_M(\alpha ).$$
Letting $M\to \infty$, we have
$$\dim_{\mathrm{H}}\overline{E}(\alpha )\ge s(\alpha).$$

\section*{Acknowledgments}
We thank Professor Lingmin Liao for this numerous helpful suggestions. Zhihui Li was partially supported by Natural Science Foundation of Hubei Province of China 2022CFC013.

\end{document}